\documentclass[ 
a4paper
,11pt
]{amsart}

\usepackage{amssymb,latexsym}
\usepackage[latin9]{inputenc}
\usepackage{amsmath}
\usepackage{amsfonts}
\usepackage{dsfont}
\usepackage{amsthm}
\usepackage{thmtools}

\declaretheorem[parent=section]{theorem}
\declaretheorem[sibling=theorem]{proposition}
\declaretheorem[sibling=theorem]{lemma}
\declaretheorem[sibling=theorem]{corollary}

\declaretheorem[sibling=theorem,style=remark]{remark}
\declaretheorem[sibling=theorem,style=remark]{example}
\declaretheorem[sibling=theorem,style=definition]{definition}

\declaretheorem[numbered=no,name=Theorem]{theorem*}
\declaretheorem[numbered=no,style=definition,name=Definition]{definition*}
\declaretheorem[numbered=no,name=Proposition]{proposition*}

\usepackage{enumerate}
\usepackage{ifthen}
\usepackage{comment}
\usepackage{xcolor}
\usepackage{mathtools}
\usepackage{booktabs}
\usepackage{etoolbox}
\usepackage[T1]{fontenc}

\usepackage{hyperref}
\hypersetup{
   backref
   ,colorlinks=true
   ,allcolors=black
}

\begin{document}
\bibliographystyle{alpha}

\newcommand{\mi}{\ensuremath{\mathrm{i}}}
\newcommand{\me}{\ensuremath{\mathrm{e}}}
\newcommand{\mPS}{\ensuremath{(\Omega,\mathcal{F},P)}}
\newcommand{\Id}{\on{Id}}
\newcommand{\probSpace}{\ensuremath{(\Omega,\mathcal{F},P)}}
\newcommand{\fprobSpace}{\ensuremath{(\Omega,\mathcal{F}_t,P)}}
\newcommand{\prob}{\ensuremath{\operatorname{P}}}
\newcommand{\fps}{\ensuremath{\left( \Omega,\mathcal{F},(\mathcal{F}_t)_{t
        \geq 0}, \prob \right)}}
\newcommand{\stset}{\ensuremath{\mathcal{T}}}
\newcommand{\ind}{1{\hskip -2.5 pt}\hbox{\textnormal{I}}}

\newcommand{\on}[1]{\operatorname{#1}}
\newcommand{\abs}[1]{\left\vert #1 \right\rvert}
\newcommand{\norm}[1]{\left\lVert #1 \right\rVert}
\newcommand{\Ex}[1]{\operatorname{E} \left[ #1 \right]}
\newcommand{\qvar}[2]{\langle #1, #2 \rangle}
\newcommand{\Set}[1]{\left\{ #1 \right\}}
\newcommand{\diff}[1]{\operatorname{d}\ifthenelse{\equal{#1}{}}{\,}{#1}}

\newcommand{\E}{\operatorname{E}}
\newcommand{\C}{\mathbb{C}}
\newcommand{\R}{\mathbb{R}}
\newcommand{\Q}{\mathbb{Q}}
\newcommand{\N}{\mathbb{N}}
\newcommand{\Z}{\mathbb{Z}}
\newcommand{\LL}{\on{L}}

\newcommand{\Cum}{\mathcal{K}}

\newtheorem*{ld_thm:mix}{Theorem \ref{ld_thm:mix}}

\author{Simon Campese}
\email{campese@mat.uniroma2.it}
\address{Universit\`a degli Studi di Roma Tor Vergata, Dipartimento di
  Matematica, Via della Ricerca Scientifica 1, 00133 Roma, Italy}
\title[Complex Fourth Moment Theorems]{Fourth Moment Theorems for complex
  Gaussian approximation} 
\date{\today}
\thanks{Research was supported by ERC grant 277742 Pascal}
\begin{abstract}
  {
  We prove a bound
  for the Wasserstein distance between vectors of smooth complex random
  variables and complex Gaussians in the framework of complex Markov diffusion
  generators. For the special case of chaotic   
  eigenfunctions, this bound can be expressed in terms of certain fourth
  moments of the vector, yielding a quantitative Fourth Moment Theorem for
  complex Gaussian approximation on complex Markov diffusion chaos. This extends
  the results of~\cite{azmoodeh_fourth_2014}
  and~\cite{campese_multivariate_2015} for the real case. Our main ingredients
  are a complex version of the so called $\Gamma$-calculus and Stein's method
  for the multivariate complex Gaussian distribution.} 
\end{abstract}
\maketitle

\section{Introduction}
\label{s-sect-subs-calc}

Let $F_n$ be a sequence of random variables and $\gamma$ be a probability
measure. A Fourth Moment Theorem for $F_n$ holds, if there exists a
polynomial $P=P(m_1(F_n),\dots,m_4(F_n))$ in the first four moments of $F_n$
such that $P(m_1(F_n),\dots,m_4(F_n)) \to 0$ characterizes, or at least implies
the convergence of the laws of $F_n$ towards $\gamma$. The first discovery of
such a Fourth Moment Theorem dates back to 2005, when Nualart and Peccati, in
their seminal paper~\cite{nualart_central_2005}, characterized 
convergence in distribution of a sequence $F_n$ of multiple Wiener-It\^o
integrals towards a Gaussian distribution by the convergence of the moment
sequence $m_4(F_n)-3m_2(F_n)^2$. The next developments were an extension to the
multidimensional case by Peccati and Tudor in~\cite{peccati_gaussian_2005}, the
introduction of Malliavin calculus by Nualart and Ortiz-Latorre
in~\cite{nualart_central_2008} and, starting in~\cite{nourdin_steins_2009} by
Nourdin and Peccati and then followed by numerous other contributions of the
same two authors and their collaborators, the combination of Malliavin calculus
and Stein's method to obtain quantitative central limit theorems in strong
distances. This Malliavin-Stein method has found widespread applications, for
example in statistics, mathematical physics, stochastic geometry or free
probability {(see the webpage~\cite{steinmalliavin} for an
  overview)}. For a textbook 
introduction to the method, we refer to~\cite{nourdin_normal_2012}.

Recently, initiated by Ledoux in~\cite{ledoux_chaos_2012} and then further
refined by Azmoodeh, Campese and Poly in~\cite{azmoodeh_fourth_2014}, it was
shown that Fourth Moment Theorems hold in the very general framework of Markov
diffusion generators (see~~\cite{bakry_analysis_2014} for an
{exhaustive} treatment 
of this framework), in which the aforementioned setting of multiple
Wiener-It\'o integrals appears as the special case of the Ornstein-Uhlenbeck 
generator. This abstract point of view not only allows for a drastically
simplified proof of Nualart and Peccati's classical Fourth Moment Theorem, but
also provides 
new Fourth Moment Theorems for previously uncovered structures such as Laguerre
and Jacobi chaos and {new} target laws (such as the Beta
distribution). We refer  
to~\cite{azmoodeh_fourth_2014} for details. Stein's method continues
to work in this framework as well {and can be used to associate
  quantitative estimates to these results}. The 
abstract framework also led to new results for the Ornstein-Uhlenbeck setting,
such as an ``even-moment theorem'' (see~\cite{azmoodeh_generalization_2013}) and
advances regarding the Gaussian product conjecture
(see~\cite{malicet_squared_2015}).   

In this paper, we extend this abstract diffusion generator framework to the
complex case, thus covering complex valued random variables.
Our main result is a quantitative bound between the Wasserstein distance of a
multivariate complex Gaussian random vector and vectors of square integrable
complex random variables in the domain of a carr\'e du champ operator associated
to a diffusive Markov generator (all unexplained terminology will be
{introduced} 
below). For the case of chaotic 
eigenfunctions, this bound can be expressed purely in terms of the second and
fourth absolute moments of the vector and thus yields a quantitative complex
Fourth Moment Theorem. To obtain these results, we extend Stein's method to
cover the complex Gaussian distribution and develop a complex version of the so
called $\Gamma$-calculus, in which a central role will be played by the Wirtinger
derivatives $\partial_z$ and $\partial_{\overline{z}}$.

Of course, from a purely algebraic point of view, our approach is 
equivalent to the multidimensional real case, which has been treated
in~\cite{campese_multivariate_2015}, and indeed one could handle
sequences of complex valued random variables by separating real and
imaginary part and considering them as two-dimensional real vectors. For the
complex Ornstein-Uhlenbeck generator, such an ad-hoc strategy has been followed
in~\cite{chen_fourth_2014} and~\cite{chen_product_2014}.
However, by staying completely inside the complex domain, our approach has the
advantage of making available many powerful tools connected to concepts such as
holomorphy from complex analysis which have no equivalents in the real
case. Although not needed for an abstract derivation of our main results, these
tools might be useful in the future, also when taking the reverse route and
trying to prove results for the real case by translating them to the complex 
domain. {Even for the special case of the complex Ornstein-Uhlenbeck
  generator, there is much structure present in the complex domain, such as a
  fine decomposition of the eigenspaces as shown in Example~\ref{ex:1} or the
  unitary group from complex White noise analysis
  (see~\cite[Ch. 7]{hida_brownian_1980}), which has no real counterparts. }

Complex random variables are encountered naturally in many applications and
indeed this paper was motivated by the study of spin random fields arising in  
{ cosmology (see for
  example~\cite{baldi_representation_2014,geller_spin_2010,malyarenko_invariant_2013,marinucci_random_2011}). A
  followup paper with an application of our results to this case is in
  preparation.}   

The rest of this paper is organized as follows. In
Section~\ref{s-sect-calc-z=x+mi}, we introduce the notation used throughout the
rest of this paper and provide some necessary background material. The abstract
complex diffusion generator framework is presented in
Section~\ref{s-as-real-case}{, while} our main results are stated
and proved in  
Section~\ref{s-now-let-z}.

\section{Preliminaries}
\label{s-sect-calc-z=x+mi}

\subsection{Wirtinger calculus}
\label{s-sect-mall-oper}
Let $x,y\in \R$ be two real variables and define the complex variables $z=x +
\mi y$ and $\overline{z} = x-\mi y$, where $\mi^2=-1$. Then, every function
$\widetilde{f}=\widetilde{f}(x,y) \colon \R^{2} \to \C$ can be considered as a
function $f=f(z) \colon \C \to \C$ via the identity
\begin{equation}
  \label{eq:7}
  \widetilde{f}(x,y) = \widetilde{f} \left( \frac{z+\overline{z}}{2},
    \frac{z-\overline{z}}{2} \right) =  f(z).
\end{equation}
Conversely, {every} function $f=f(z) \colon \C \to \C$ can be considered
as a function $\widetilde{f}=\widetilde{f}(x,y) \colon \R^{2} \to \C$ by writing
\begin{equation}
  \label{eq:10}
  f(z) = f(x+ \mi y) = \widetilde{f}(x,y).
\end{equation}
With this notation, the Wirtinger derivatives are defined as
\begin{equation*}
  \partial_z = \frac{1}{2} \left( \partial_x - \mi \partial_y \right)
  \qquad \text{and} \qquad
  \partial_{\overline{z}} = \frac{1}{2} \left( \partial_x + \mi \partial_y
  \right),
\end{equation*}
where, here and in the following, we use the shorthand
$\partial_{a} = \frac{\partial}{\partial a}$, and, more generally,
$\partial_{a_1 a_2 \dots a_d} =
\frac{\partial^d}{\partial_{a_{1}} \partial_{a_2} \dots \partial_{a_d}}$ to
denote derivatives with respect to the variables $a_1,\dots,a_d$ (the concrete
interpretation as a partial or Wirtinger derivative { will always be
clear from the context}). 
Starting from the Wirtinger derivatives $\partial_z$ and $\partial_{\overline{z}}$,
one can get back the partial derivatives $\partial_x$ and $\partial_y$ by the
identities 
\begin{equation*}
  \partial_x = \partial_z + \partial_{\overline{z}} 
  \qquad \text{and} \qquad
  \partial_y = \mi \left( \partial_z - \partial_{\overline{z}} \right).
\end{equation*}
It is straightfoward to see that both Wirtinger derivatives are
linear and merit their names by satisfying the derivation properties (product rules)
\begin{equation*}
  \partial_z (fg) = (\partial_z f) g + f \partial_z g
  \qquad \text{and} \qquad
  \partial_{\overline{z}} (fg) = (\partial_{\overline{z}} f) g + f \partial_{\overline{z}} g,
\end{equation*}
Moreover, the conjugation identities
\begin{equation*}
  \overline{\partial_z f} = \partial_{\overline{z}} \overline{f}
  \qquad \text{and} \qquad
  \overline{\partial_{\overline{z}} f } = \partial_z \overline{f}.
\end{equation*}
hold. The chain rules take the form 
\begin{align*}
  \partial_z (f \circ g) &= \left((\partial_z f) \circ g \right) \partial_z g +
  \left( \left( \partial_{\overline{z}} f \right) \circ g  \right) \partial_z
  \overline{g}, \\
  \partial_{\overline{z}} (f \circ g)
                         &=
                           \left( \left( \partial_z f \right) \circ g
                           \right) \partial_{\overline{z}} g
                           +
                           \left(  \left( \partial_{\overline{z}} f \right)
                           \circ g 
                           \right) \partial_{\overline{z}} \overline{g}.
\end{align*}
In particular, we see that $\partial_z z
= 1$, $\partial_z \overline{z}=0$, $\partial_{\overline{z}} \overline{z}=1$ and 
$\partial_{\overline{z}} z = 0$, which, in view of the chain and product rule,
allows to formally treat $z$ and $\overline{z}$ as if they were independent
variables when differentiating. {Heuristically}, when applying the
Wirtinger derivatives to a function $f \colon \C \to \C$, one does not need to
consider $f(z)$ as a function 
$\widetilde{f}(x,y)$ and then compute the partial derivatives with respect to
$x$ and $y$, but can instead directly apply the formal rules of differentiation
to the complex variables $z$ and $\overline{z}$. For example, we have 
for $p,q \neq -1$ that 
$\partial_z (z^p \overline{z}^q)=pz^{p-1}\overline{z}^q$ and
$\partial_{\overline{z}} (z^p\overline{z}^q) = qz^p\overline{z}^{q-1}$.

In the sequel, we will work in general dimension $d \geq 1$, considering 
functions $f \colon \C^d \to \C$ and $\widetilde{f} \colon \R^{2d} \to \C$ which
continue to be related through~\eqref{eq:7} and~\eqref{eq:10}, where now
$x=(x_1,\dots,x_d)$, $y=(x_1,\dots,x_d)$ and $z=(z_1,\dots,z_d)$ are vectors of
variables. For these functions, we define the gradients $\nabla f$ and
$\overline{\nabla}f$ by  
\begin{equation*}
  \nabla f = \left( \partial_{z_1} f, \partial_{z_2} f, \dots, \partial_{z_d} f
  \right)^{T}
\end{equation*}
and
\begin{equation*}
  \overline{\nabla} f = \left( \partial_{\overline{z}_1}
    f, \partial_{\overline{z}_2} f, \dots, \partial_{\overline{z}_d} f \right)^T,
  \end{equation*}
and the complex Hessians $\nabla\nabla f$,
$\overline{\nabla\nabla} f$, $\overline{\nabla} 
\nabla f$ and $\nabla \overline{\nabla} f$ by
\begin{align*}
  \nabla\nabla f &= \left( \partial_{z_jz_k} f
                                              \right)_{1 \leq j,k \leq d}, 
                   &
  \overline{\nabla\nabla} f &= \left( \partial_{\overline{z}_j \overline{z}_k} f
                              \right)_{1 \leq j,k \leq d},
  \\                              
  \overline{\nabla} \nabla f &= \left( \partial_{\overline{z}_j z_k} f \right)_{1
                               \leq j,k \leq d}
                               &\text{and \qquad}
  \nabla \overline{\nabla} f &= \left( \partial_{z_j \overline{z}_k} f \right)_{1
  \leq j,k \leq d}. 
\end{align*}
With some abuse of notation, we say that $f$ is an element of $\mathcal{C}^{m}(\C^d)$, if
the associated function $\widetilde{f}$ belongs to
$\mathcal{C}^m(\R^{2d})$. Similarly, a function $f$ has bounded Wirtinger
derivatives up to some order $m$, if the associated function~$\widetilde{f}$ has
bounded partial derivatives up to this order. 

\subsection{The complex normal distribution}
\label{s-some-measure-nu}

Given a probability space $\mathcal{P}=(\Omega,\mathcal{F},P)$ and two
real-valued random variables $X$ and $Y$, the quantity $Z=X+\mi Y$ is called a
complex valued random variable. The characteristic function, law and, if it
exists, density of $Z$ are defined as being the corresponding quantities of
the two-dimensional real random vector $(X,Y)$. With the notation of the
previous subsection, we clearly have that if $f(z)$ is the density of $Z$, then
$\widetilde{f}(x,y)$ is the density of $(X,Y)$, and the analogous statement is
true for the law. From the characteristic function~$\widetilde{\rho}(\xi,\upsilon)$
of $(X,Y)$, we readily calculate the characteristic function
$\rho(\zeta)=\rho(\xi+ \mi \upsilon)$ of $Z$ to be
\begin{equation*}
  \widetilde{\rho}(\zeta) = \Ex{\me^{\mi \mathfrak{Re} (\left\langle \zeta,Z \right\rangle_{\C^d})}},
\end{equation*}
where, here and in the following, $\on{E}$ denotes mathematical expectation and
$\mathfrak{Re}(z)$ and $\mathfrak{Im}(z)$ stand for the real and imaginary parts
of a complex number~$z$.

\begin{definition}[multivariate complex normal distribution]
  \label{def:2}
  For $d\geq 1$, $\mu \in \C^d$ and a positive definite $d \times d$ Hermitian
  matrix 
  $\Sigma$, a complex random vector $Z \in \C^d$ is said to follow a \emph{multivariate
    complex normal distribution} with mean $\mu$ and covariance $\Sigma$,
  short: $Z \sim C\mathcal{N}_d(\mu,\Sigma)$, if it has the density function 
  \begin{equation}
    \label{eq:34}
  f(z) =
  \frac{1}{\pi^d \abs{\on{det} \Sigma}} \on{exp} \left( -  \overline{\left( z -
    \mu\right)}^{T} \Sigma^{-1} \left( z - \mu \right) \right),
\end{equation}
where $A^{T}$ denotes the transpose of $A$.
\end{definition}

\begin{remark}
  \label{rem-gauss}
  Let $Z \sim C\mathcal{N}_d(\mu,\Sigma)$.
  \begin{enumerate}[(i)]
  \item
  Straightforward calculations show that
   \begin{equation*}
     \Ex{ (Z-\mu) \overline{(Z_k-\mu)}^{T} } = \Sigma
   \end{equation*}
   and
   \begin{equation}
     \label{eq:35}
      \Ex{ (Z-\mu) (Z - \mu)^T} = 0.
    \end{equation}
    Furthermore, it can be shown that $Z$ is \emph{circularly symmetric}: For
    any $\alpha \in \R$, the rotated vector $\me^{\mi \alpha} Z$ has the same
    distribution as $Z$. { Each circularly symmetric complex
      Gaussian vector can be obtained via a linear transformation of a standard
      complex Gaussian vector $Z \sim C\mathcal{N}_d(0,\Id_d)$ whose real and
      imaginary part are \emph{independent} real-valued standard Gaussian
      vectors. Some authors drop the independency assumption and thus obtain
      more general complex Gaussian vectors for which the matrix on the left
      hand side of~\eqref{eq:35}, sometimes called the relation matrix, is no
      longer zero. However, when we speak of a complex Gaussian vector, we
      always mean the circularly symmetric case of Definition~\ref{def:2}.}  
  \item The characteristic function $\rho$ of $Z$ is given by
    \begin{equation*}
     \rho(\zeta) 
     =
     \exp \left( \mi \mathfrak{Re}( \left\langle \mu,\zeta \right\rangle_{\C^d}) -
       \frac{1}{4} \left\langle \Sigma \zeta,\zeta \right\rangle_{\C^d} \right),
   \end{equation*}
   which shows that $Z$ is determined by its moments, i.e. (assuming $\mu=0$ for
   notational convenience) any $d$-dimensional 
   complex random vector~$W$ satisfying
   \begin{equation*}
     \Ex{
       \prod_{j=1}^d (W_j-\mu_j)^{p_j} (\overline{W}_j-\overline{\mu}_j)^{q_j} 
     }
     =
     \Ex{
       \prod_{j=1}^d (Z_j-\mu_j)^{p_j} (\overline{Z}_j-\overline{\mu}_j)^{q_j} 
     }
   \end{equation*}
  for all $p_j,q_j \in \N_0$, $j=1,\dots,d$, has the same law as $Z$.
  \end{enumerate}
\end{remark}

We will need the following complex version of the Gaussian integration by parts
formula, which for convenience will be stated for the centered case. 

\begin{lemma}
  \label{lem:4}
  Let $Z \sim CN_d(0,\Sigma)$ and $\varphi \colon \mathbb{C}^d \to \mathbb{C}$
  such that, {for $1 \leq i \leq d$, the Wirtinger derivatives}
  $\partial_{z_i} \varphi$ and $\partial_{\overline{z}_i} \varphi$ 
  { exist and} have at most polynomial growth in $z_i$ and
  $\overline{z}_i$ respectively. Then it holds that 
  \begin{equation}
    \label{eq:15}
    \Ex{Z_i \varphi(Z_1,\dots,Z_d)} = \sum_{j=1}^d \Ex{ Z_i \overline{Z}_j}
    \Ex{ \partial_{\overline{z}_j} \varphi(Z_1,\dots,Z_d)}
  \end{equation}
  and
  \begin{equation}
    \label{eq:16}
    \Ex{\overline{Z}_i \varphi(Z_1,\dots,Z_d)} = \sum_{j=1}^d \Ex{ Z_j
      \overline{Z}_i} 
    \Ex{ \partial_{z_j} \varphi(Z_1,\dots,Z_d)}
  \end{equation}  
\end{lemma}

\begin{proof}
  The proof is standard and straightforward but is included nevertheless to
  { demonstrate the use of Wirtinger calculus}. We will 
  prove formula~\eqref{eq:15} and note that~\eqref{eq:16} then follows by conjugation.  
  As $\Sigma$ is positive definite Hermitian, it admits a normal square root
  $A$ such that $\Sigma=A^{\ast}A$, where $A^{\ast}$ denotes the conjugate
  transpose of $A$. Clearly, both $A$ and $A^{\ast}$ are
  invertible and we have that $\Sigma^{-1} = (A^{\ast})^{-1} A^{-1}$. The linear
  transformation 
  \begin{equation*}
    \xi = A^{-1} z
  \end{equation*}
  induces a linear transformation on $\R^{2d}$ whose constant volume element
  will be 
  denoted by $v$. Writing $z=x+iy$ and $\zeta=\xi+\mi \upsilon$ (with
  $z_i=x_i+\mi y_i$ and $\zeta_i=\xi_i + \mi \upsilon_i$), we get that 
  \begin{align*}
    \Ex{Z_i \varphi(Z)}
    &=
      \frac{1}{\pi^d \abs{\on{det}\Sigma}}
      \int_{\R^{2d}}^{}
      z_i \varphi(z)
      \,
      \me^{- z^{\ast} \Sigma^{-1} z}
      \diff{(x,y)}
    \\ &=
         \frac{v}{\pi^d \abs{\on{det}\Sigma}}
         \sum_{j=1}^d A_{ij}
         \int_{\R^{2d}}^{}
          \zeta_j \varphi(A \zeta)
      \,
      \me^{- \zeta^{\ast} \zeta}
         \diff{(\xi,\upsilon)}
    \\ &= 
         \frac{- v}{\pi^d \abs{\on{det}\Sigma}}
         \sum_{j=1}^d A_{ij}
         \int_{\R^{2d}}^{}
         \varphi(A \zeta)
         \left(
         \partial_{\overline{\zeta}_j} \me^{- \zeta^{\ast} \zeta} \right)
         \diff{(\xi,\upsilon)}.
  \end{align*}
  By the product rule, it holds that
  \begin{equation*}
    \varphi(A \zeta)
         \left(
           \partial_{\overline{\zeta}_j} \me^{- \zeta^{\ast} \zeta} \right)
         =
        \partial_{\overline{\zeta}_j}  \left( \varphi(A \zeta)
          \me^{- \zeta^{\ast} \zeta} \right)
        -
        \left( \partial_{\overline{\zeta}_j}     \varphi(A \zeta)
        \right)
         \me^{- \zeta^{\ast} \zeta}.
 \end{equation*}
 Now, by a Fubini argument, 
 \begin{equation*}
   \int_{\R^{2d}}^{}  \partial_{\overline{\zeta}_j}  \left( \varphi(A \zeta)
          \me^{- \zeta^{\ast} \zeta} \right) \diff{(\xi,\upsilon)} = 0.
 \end{equation*}
 Furthermore, by the chain rule,
 \begin{equation*}
   \partial_{\overline{\zeta}_j}     \varphi(A \zeta)
   =
   \sum_{k=1}^d
   \overline{A}_{k,j}
   (\partial_{\overline{\zeta}_k} \varphi) (A \zeta).
 \end{equation*}
 Therefore,
 \begin{align*}
   \Ex{Z_i \varphi(Z)}
   &=
   \frac{v}{\pi^d \abs{\on{det}\Sigma}}
         \sum_{j,k=1}^d A_{ij} \overline{A}_{k,j}
         \int_{\R^{2d}}^{}
         (\partial_{\overline{\zeta}_k} \varphi)
         (A \zeta)
         \me^{- \zeta^{\ast} \zeta}
     \diff{(\xi,\upsilon)}
   \\ &=
   \frac{1}{\pi^d \abs{\on{det}\Sigma}}
         \sum_{j,k=1}^d A_{ij} \overline{A}_{k,j}
         \int_{\R^{2d}}^{}
         \partial_{\overline{z}_k} \varphi
         (z)
         \me^{- z^{\ast} z}
        \diff{(x,y)}.
 \end{align*}
 Noting that $\sum_{j=1}^d A_{i,j} \overline{A}_{k,j} = \sum_{j=1}^d A_{i,j}
 A^{\ast}_{j,k} = \Sigma_{i,k} = \Ex{Z_i \overline{Z}_k}$
 proves~\eqref{eq:15}.
\end{proof}

As an immediate consequence of Lemma~\ref{lem:4}, we see that for $Z\sim
C\mathcal{N}_d(0,\Sigma)$ and all multi-indices $p=(p_1,\dots,p_d) \in \N_0^d$ of order 
at least one it holds that
\begin{equation*}
  \Ex{ \prod_{j=1}^d Z_j^{p_j}} = \Ex{ \prod_{j=1}^d \overline{Z}_j^{p_j}} = 0.
\end{equation*}
Furthermore, for the case $Z \sim C\mathcal{N}_1(0,\sigma^2)$, Lemma~\ref{lem:4} 
yields the well-known moment-formula
\begin{equation*}
  \begin{cases}
    p! \, \sigma^{2p} &\qquad \text{if $p=q$} \\
    0  &\qquad \text{if $p \neq q$},
  \end{cases}
\end{equation*}
valid for $p,q \in \N_0$ (with the usual convention that $0!=1$).

\subsection{Stein's method for the complex normal distribution} 
\label{s-sect-mall-oper-1}

For a quadratic matrix $A$, the Hilbert-Schmidt norm $\norm{A}_{\text{HS}}$ is
defined via the inner product 
$\left\langle A,B \right\rangle_{\text{HS}} = \on{tr}(A \overline{B}^T)$.

The next lemma is an adaptation of the Stein characterization for the
multivariate real normal distribution (see~\cite[Lemma
2.1]{chatterjee_multivariate_2008}, \cite[Lemma 3.3 
]{nourdin_multivariate_2010} and also~\cite[Lemma
2.6]{reinert_multivariate_2009}) to the complex case. Recall the definitions of
complex gradients and Hessians given in 
Section~\ref{s-sect-mall-oper}. 

\begin{lemma}[Stein's Lemma for the complex Gaussian distribution]
  \label{lem:3}
  For $d\geq 1$, let $\Sigma$ be a positive definite, Hermitian matrix and $Z
  \sim  C\mathcal{N}_d(0,\Sigma)$.  
  \begin{enumerate}[i)]
  \item A $d$-dimensional complex random vector $Y$ has the complex normal
    distribution $C\mathcal{N}_d(0,\Sigma)$, if, and only if, the identity
    \begin{multline*}
            \Ex{ \left\langle \overline{\nabla} \nabla f(Y), \overline{\Sigma}
        \right\rangle_{\text{HS}}}
      +
      \Ex{ \left\langle \nabla \overline{\nabla} f(Y), \Sigma
        \right\rangle_{\text{HS}}}
      \\ -
      \Ex{
        \left\langle \nabla f(Y), \overline{Y} \right\rangle_{\C^d}}
      -
       \Ex{
         \left\langle \overline{\nabla} f(Y), Y \right\rangle_{\C^d}}
       = 0
     \end{multline*}
     holds for any $f \in \mathcal{C}^2(\C^d)$ which satisfies
    \begin{multline*}
            \Ex{\abs{\left\langle \overline{\nabla} \nabla f(Y),
                  \overline{\Sigma} 
        \right\rangle_{\text{HS}}}}
      +
      \Ex{ \abs{\left\langle \nabla \overline{\nabla} f(Y), \Sigma
          \right\rangle_{\text{HS}}}}
      \\ +
      \Ex{
        \abs{
        \left\langle \nabla f(Y), \overline{Y} \right\rangle_{\C^d}}}
      +
       \Ex{\abs{
         \left\langle \overline{\nabla} f(Y), Y \right\rangle_{\C^d}}}
       < \infty.
     \end{multline*}
  \item Given $h \in \mathcal{C}^2(\C^d)$ with bounded derivatives up to
    order two, the function
    \begin{equation}
      \label{eq:4}
      U_h(z) = \int_0^1 \frac{1}{2t}
      \left(
        \Ex{ h(Z_{z,t})}
        -
        \Ex{h(Z)}
      \right)
      \diff{t},
    \end{equation}
    where $Z_{z,t}= \sqrt{t}z + \sqrt{1-t}Z$, is a solution to the complex Stein
    equation 
    \begin{multline}
      \label{eq:22}
      \left\langle \overline{\nabla} \nabla f(z), \overline{\Sigma}
      \right\rangle_{\text{HS}}
      +
      \left\langle \nabla \overline{\nabla} f(z), \Sigma
      \right\rangle_{\text{HS}}
      \\ -
      \left\langle \nabla f(z), \overline{z} \right\rangle_{\C^d}
      -
      \left\langle \overline{\nabla} f(z), z
      \right\rangle_{\C^d} 
      =
      h(z) - \Ex{ h(Z)}.
    \end{multline}
  \end{enumerate}
\end{lemma}

\begin{proof}
  The real counterpart of this lemma has been proven by direct calculations
  using Gaussian integration by parts in~\cite[Lemma
  2.1]{chatterjee_multivariate_2008}, and based on the    generator approach
  of~\cite{barbour_steins_1990} in~\cite[Lemma
  3.3]{nourdin_multivariate_2010}. Both of these proofs can be straightforwardly  
  adapted to the complex case (either using complex Gaussian by parts or the
  complex Ornstein-Uhlenbeck generator) which is why we omit the details.
\end{proof}

We also need the following technical lemma, which can be deduced by adapting the
proof of inequality (3.4) in~\cite{nourdin_multivariate_2010}.  

\begin{lemma}
  \label{lem:5}
  In the setting and with the notation of Lemma~\ref{lem:3}, let $c(\Sigma) =
  \norm{\Sigma^{-1}}_{\text{op}} \norm{\Sigma}_{\text{op}}^{1/2}$. Then, for any
  $\alpha \in [0,1]$, the complex
  Hessians of the Stein solution~\eqref{eq:4} satisfy the bounds
  \begin{align*}
    \norm{\nabla \overline{\nabla} U_h(z)}_{\text{HS}}
    &\leq
      c(\Sigma)                                                
      \left( \alpha \max_{1 \leq k \leq d} \norm{\partial_{z_j}h}_{\infty}  +
      (1-\alpha) \max_{1 \leq k \leq d} \norm{\partial_{\overline{z}_j}
      h}_{\infty}  \right),
    \\
    \norm{\overline{\nabla} \nabla U_h(z)}_{\text{HS}}
    &\leq
    c(\Sigma)                                                
      \left( \alpha \max_{1 \leq k \leq d} \norm{\partial_{z_j}h}_{\infty}  +
      (1-\alpha) \max_{1 \leq k \leq d} \norm{\partial_{\overline{z}_j}
      h}_{\infty}  \right),
    \\
    \norm{\nabla \nabla U_h(z)}_{\text{HS}}
    &\leq
    c(\Sigma)                                                
      \max_{1 \leq k \leq d} \norm{\partial_{z_j}h}_{\infty}
    \\
    \intertext{and}
    \norm{\overline{\nabla}\overline{\nabla} U_h(z)}_{\text{HS}}
    &\leq
    c(\Sigma)                                                
      \max_{1 \leq k \leq d} \norm{\partial_{\overline{z}_j}h}_{\infty}.
  \end{align*}
\end{lemma}

\begin{remark}\hfill
  \begin{enumerate}[i)]
  \item {The proofs of Lemmas~\ref{lem:3} and~\ref{lem:5}}
    crucially 
      depend on the characterizing property that any complex Gaussian vector $Z
      \sim 
    C\mathcal{N}_d(0,\Sigma)$ can be obtained via a linear transformation of a
    standard complex Gaussian vector $\widetilde{Z}\sim
    C\mathcal{N}_d(0,\Id_d)$. An adaptation of these proofs to the larger class
    of not necessarily circularly symmetric complex Gaussian vectors hinted at
    {in Remark~\ref{rem-gauss}} is thus not possible.
  \item As $\Sigma$ and its inverse are positive definite and Hermitian, the
    operator norms of these two matrixes coincide with their spectral
    radii. Thus, the constant $c(\Sigma)=\norm{\Sigma^1}_{op}
    \norm{\Sigma}_{op}^{1/2}$ appearing in the bounds for $U_h$ in
    Lemma~\ref{lem:5} coincides  with 
    $\sqrt{\lambda_{max}}/\lambda_{min}$ where $\lambda_{max}$ and
    $\lambda_{min}$ denote the largest and smallest eigenvalues of $\Sigma$,
    respectively. In particular, in the case $d=1$, where $\Sigma=\lambda > 0$, this
    constant becomes $1/\sqrt{\lambda}$. 
  \end{enumerate}
\end{remark}

For the case $d=1$, it is actually possible to derive a simpler characterizing
differential equation.

\begin{lemma}
  \label{lem:1}
  A complex valued random variable $Z$ has the standard complex normal
  distribution if, and only if 
  \begin{equation}
    \label{eq:8}
    \Ex{\partial_z f(Z)} - \Ex{\overline{Z} f(Z)} = 0
  \end{equation}
  for all { Wirtinger differentiable functions $f \colon \C \to \C$
    such that $\partial_zf$ has at most polynomial growth.} 
\end{lemma}

\begin{proof}
  Necessity of condition~\eqref{eq:8} is implied by Gaussian integration by parts
  {(Lemma~\ref{lem:4})}. {Sufficiency follows by}
  inserting the polynomials  $f(z)=\overline{z}^pz^q$, {which}
  immediately yields the moment recursion for the  standard complex
  Gaussian{, and noting that the complex Gaussian distribution is
    determined by its moments.}   
\end{proof}

\begin{remark}
From identity~\eqref{eq:8}, one is led to the ``Stein equation''
\begin{equation}
  \label{eq:9}
  \partial_z f(z) - \overline{z} f(z) = h(z) - \Ex{h(Z)},
\end{equation}
where $h  \colon \C \to \C$ is some given function and $Z$ has the standard
complex Gaussian distribution. Note that if we formally replace the complex
variable $z$ with a real variable $x$ {and the Wirtinger derivative
  $\partial_z$ with the partial derivative $\partial_x$}, we obtain the
classical Stein equation of 
the one-dimensional Gaussian distribution. However, as one sees after writing
$z=x+\mi y$ {and} separating real and imaginary parts, this equation
is not solvable 
in general. This was to be expected, as otherwise, using Stein's method for the
real case, one could obtain bounds in total variation and Kolmogorov distance
for two-dimensional real Gaussian approximation, which is not possible using
this approach (see for
example~\cite[pp.263]{chatterjee_multivariate_2008}. Lemma~\ref{lem:1} can thus 
not be quantified.
\end{remark}

\section{Complex Markov diffusion generators}
\label{s-as-real-case}

As in the real case, we start with a \emph{good measurable space} $E$ in the sense
of~\cite[p.7]{bakry_analysis_2014} (for example, take $E$ to be a Polish space),
equipped with a probability measure $\mu$. On $L^2(E,\R,\mu)$, let $\LL$ be a
symmetric Markov diffusion generator $\LL$ with discrete spectrum $S =
\Set{ - \lambda_{k}}$, where the eigenvalues $-\lambda_k$ are ordered by magnitude,
i.e. $0=\lambda_0 < \lambda_1 < \dots$. In the language of functional analysis,
$-\LL$ is a positive, self-adjoint linear operator vanishing on the
constants. The associated bilinear carr\'e du champ operator $\Gamma$, acting on
a set $\mathcal{A}_0$ which we assume to be dense in $L^p(E,\R,\mu)$ for all $p
\geq 1$, is defined in the usual way as
\begin{equation}
  \label{eq:18}
  2\Gamma(U_1,U_2) = \LL (U_1U_2)- U_{1} \LL U_2 - U_2 \LL U_1
\end{equation}
and for any
smooth function $\varphi \colon \R^d \to \R$ the diffusion property
\begin{multline}
  \label{eq:14}
  \LL \varphi (U) = \LL \varphi(U_1,\dots,U_d) \\ = \sum_{k=1}^d \partial_{x_k} 
  \varphi(U_1,\dots,U_d) \LL U_k + \sum_{j,k=1}^d \partial_{x_j x_k} \varphi(U_1,\dots,U_d)
  \Gamma(U_j,U_k) 
\end{multline}
holds. As is well known, this diffusion property is equivalent to the chain rule
\begin{equation}
  \label{eq:30}
  \Gamma(\varphi(U_1,\dots,U_d),V) = \sum_{j=1}^d \partial_j \varphi(U_1,\dots,U_d) \Gamma(U_k,V).
\end{equation}
Through straightfoward complexification, we can extend $\LL$ and $\Gamma$ to { act
on} the space~$L^2(E,\C,\mu)$. Writing $F=U+\mi V$ for a generic element of
$L^2(E,\C,\mu)$, this extension of $\LL$, which we denote for
the moment by $\widehat{\LL}$, is simply defined as
\begin{equation}
  \label{eq:20}
  \widehat{\LL} =
  \widehat{\LL}(U + \mi V) = \LL U + \mi \LL V.
\end{equation}
We immediately see that $-\widehat{\LL}$ remains a positive, self-adjoint
operator vanishing on constants and that its spectrum coincides with the one 
of $\LL$. Indeed, assuming that $U+\mi V$ is an eigenfunction of $\widehat{\LL}$
with some eigenvalue $\lambda$, we have that
\begin{equation*}
  \lambda (U + \mi V) = \widehat{\LL} (U + \mi V) = \LL U + \mi \LL V.
\end{equation*}
Comparing imaginary and real parts, this gives that $\lambda$ lies in the
spectrum of $\LL$ and that both, $U$ and $V$ are eigenfunctions of $\LL$ with eigenvalue $\lambda$.
Similarly, for $U_1,U_2,V_1,V_2 \in \mathcal{A}_0$, the extended carr\'e du
champ operator $\widehat{\Gamma}$, defined on $\widehat{\mathcal{A}}_0 =
\mathcal{A}_0 \oplus \mi \mathcal{A}_0$, is given by
\begin{equation}
  \label{eq:26}
  \widehat{\Gamma}(U_1+\mi V_1,U_2 + \mi V_2) = \Gamma(U_1,U_2) +
  \Gamma(V_1,V_2) + \mi \left( \Gamma(V_1,U_2) - \Gamma(U_1,V_2) \right).
\end{equation}
It follows that  $\widehat{\Gamma}$ is sesquilinear, positive ($\widehat{\Gamma}(F,F) \geq
0$) and Hermitian
($\widehat{\Gamma}(F,G)= \overline{\widehat{\Gamma}(G,F)}$). Furthermore, using 
identity~\eqref{eq:18}, we get that 
\begin{equation}
  \label{eq:19}
  2\widehat{\Gamma}(F,G) = \widehat{\LL}(F\overline{G}) - F \widehat{\LL}\overline{G} - \overline{G}
  \widehat{\LL} F,
\end{equation}
which yields the integration by parts formula
\begin{equation*}
  \int_E^{} \widehat{\Gamma}(F,G) \diff{\mu} = - \int_E^{} F \widehat{\LL} \,
  \overline{G} \diff{\mu}.
\end{equation*}
As $-\widehat{\LL}$ is
positive, it holds that $\overline{\widehat{\LL} F} = \widehat{\LL}
\overline{F}$. In particular, all eigenspaces of $\widehat{\LL}$ are closed under complex
conjugation and, if $\pi_j$ denotes the orthogonal projection onto
$\ker \left(\widehat{L}+ \lambda_k \Id \right)$, we have for all $F \in
L^2(E,\C,\mu)$ that $\overline{\pi_j(F)} = \pi_j(\overline{F})$. Furthermore, by
the defining equation~\eqref{eq:19} of $\Gamma$, we also see that
$\overline{\widehat{\Gamma}(F,G)}=\widehat{\Gamma}(\overline{F},\overline{G})$.
Using~\eqref{eq:20} and~\eqref{eq:26}, it is straightforward to verify the
diffusion property
\begin{align}
  \label{eq:28}
  \widehat{\LL} \varphi(F) &= \widehat{\LL} \varphi(F_1,\dots,F_d)
  \\ &= \notag
       \sum_{j=1}^d \left( \partial_{z_j} \varphi(F) \widehat{\LL} F_j
      + \partial_{\overline{z}_j} \varphi(F) \widehat{\LL} \overline{F}_j \right)
  \\ \notag &\qquad \qquad+ \sum_{j,k=1}^d \left( \partial_{z_j z_k} \varphi(F)
              \widehat{\Gamma}(F_j, 
    \overline{F}_k) + \partial_{\overline{z}_j \overline{z}_k} \varphi(F)
              \widehat{\Gamma}(\overline{F}_j,F_k) 
  \right)
  \\ \notag &\qquad \qquad+ \sum_{j,k=1}^d \left( \partial_{z_j \overline{z}_k}
              \varphi(F) 
    \widehat{\Gamma}(F_j,F_k) 
    + \partial_{\overline{z}_j z_k} \varphi(F)
              \widehat{\Gamma}(\overline{F}_j,\overline{F}_k) 
  \right),
\end{align}
and the chain rule
\begin{multline}
  \label{eq:29}
  \widehat{\Gamma}(\varphi(F),G) = \widehat{\Gamma}(\varphi(F_1,\dots,F_d))
  \\ = \sum_{j=1}^d \left( \partial_{z_j} \varphi(F) \widehat{\Gamma}(F_j,G)
    + \partial_{\overline{z}_j} \varphi(F) \widehat{\Gamma}(\overline{F}_j,G)
  \right), 
\end{multline}
both valid for smooth functions $\varphi \colon \C^d \to
\C$ and $F=(F_1,\dots,F_d) \in \widehat{\mathcal{A}}_0^d$. Here, 
$\partial_z$, $\partial_{\overline{z}}$, $\partial_{zz}$ etc. denote the
(iterated) Wirtinger derivatives introduced in
Section~\ref{s-sect-mall-oper}. To be more clear, 
the diffusion property~\eqref{eq:28} and the chain rule~\eqref{eq:29} can be
translated into versions of their real counterparts (i.e. identities~\eqref{eq:14}
and \eqref{eq:30}), by simply writing 
$\varphi(z_1,\dots,z_d)=u(x_1,\dots,x_d,y_1,\dots,y_d)+ \mi
v(x_1,\dots,x_d,y_1,\dots,y_d)$, where $z_j=x_j + \mi y_j$ and the functions $u$
and $v$ are real valued, decomposing the vector $F$ into real and imaginary
parts 
and writing the Wirtinger derivatives in terms of derivatives with respect to
the real variables $x_j$ and $y_j$. 
Because of this, we also see that, as in the real case, the diffusion property of
$\widehat{\LL}$ is equivalent to the chain rule of $\widehat{\Gamma}$.
Of course, using the fact that $\widehat{\Gamma}$ is Hermitian, we can
derive a chain rule for the second argument:
\begin{multline*}
  \widehat{\Gamma}(F,\varphi(G))
  = \widehat{\Gamma}(F,\varphi(G_1,\dots,G_d))
  \\ = \sum_{j=1}^d \left( \partial_{\overline{z}_j} \overline{\varphi}(G)
    \widehat{\Gamma}(F,G_j) 
    + \partial_{z_j} \overline{\varphi}(G) \widehat{\Gamma}(F,\overline{G}_j)
  \right). 
\end{multline*}

We are now ready to define a complex Markov diffusion generator.

\begin{definition}
  Given a good measurable space $(E,\mathcal{F})$, equipped with a probability
  measure $\mu$, a self-adjoint, linear operator $\widehat{\LL}$ acting on
  $L^2(E,\C,\mu)$ is called a \emph{complex symmetric Markov diffusion
    generator} with invariant measure $\mu$, if $-\LL$ is positive, $\LL 1 = 0$
  and the diffusion property~\eqref{eq:28} holds.
\end{definition}

Note that the construction above also works in reverse: Given a complex
symmetric Markov diffusion generator $\widehat{\LL}$, we obtain a corresponding
generator on $L^2(E,\R,\mu)$. From this abstract point of view, the real and 
complex approaches are thus completely equivalent. We state this as a short
proposition. 

\begin{proposition}
  \label{prop:1}
  $\widehat{\LL}$ is a complex Markov diffusion generator, if, and only if,
  there exists a real Markov diffusion generator   $\LL$ with (the same)
  spectrum such that $\widehat{\LL}F = \LL   \mathfrak{Re}(F) + \mi \LL
  \mathfrak{Im}(F)$ for all $F \in \on{dom}\widehat{\LL}$. 
\end{proposition}

Note that $\widehat{\LL}$ and $\widehat{\Gamma}$ coincide with $\LL$ and
$\Gamma$ when restricted to real valued arguments. Because of this and
Proposition~\ref{prop:1} we will from now on no longer notationally distinguish
$\LL$ and $\Gamma$ from their complexified versions $\widehat{\LL}$ and
$\widehat{\Gamma}$ and instead denote both versions by $\LL$ and $\Gamma$,
respectively. 

As already mentioned in the introduction, it should be noted that although
everything is equivalent from an abstract point of view, the complex case
often introduces  features absent from the real case when turning to
special cases such as the complex Ornstein-Uhlenbeck generator (see the next
example). Furthermore, in many applications it is often much more natural
to use complex Gamma-calculus, taking the direct route through the complex 
domain instead of making a detour through $\R^2$.    

\begin{example}[The complex Ornstein-Uhlenbeck generator]
  \label{ex:1}
  {
  Starting from a standard isonormal Gaussian process framework (see for
  example~\cite[Ch. 1]{nualart_malliavin_2006}
  or~\cite[Ch. 2]{nourdin_normal_2012}) for the real-valued,
  infinite-dimensional Ornstein-Uhlenbeck generator, we obtain the complex
  Ornstein-Uhlenbeck generator $\LL_{\text{OU}}$ by carrying out the
  construction outlined above (i.e. via
  Proposition~\ref{prop:1}). Equivalently, one can directly start from an
  isonormal complex Gaussian process (see~\cite{chen_fourth_2014}). 
  The generator $\LL_{\text{OU}}$ can then be decomposed in the form
  $\LL_{\text{OU}}= - \delta D$, where $\delta$ and $D$ are the complex
  Malliavin divergence and Malliavin derivative operators
  (see~\cite[Ch. 15]{janson_gaussian_1997} for more details on 
  complex Malliavin calculus). The carr\'e du champ operator
  $\Gamma_{\text{OU}}$ takes the form 
  $\Gamma_{\text{OU}}= \left\langle DF,DG \right\rangle_H$, where the inner
  product is now taken in a complex Hilbert space $H$ (coming from the
  underlying complex isonormal Gaussian process) and $D$ denotes the complex
  Malliavin derivative. The spectrum of $\LL_{\text{OU}}$ is $-\N_0$ and by our
  findings 
  above we know that all eigenfunctions $F_\lambda \in \ker \left(
    \LL_{\text{OU}}+ \lambda \Id \right)$, $\lambda \in \N_0$, are of the form
  $F_\lambda=U_\lambda + 
  \mi V_{\lambda}$, where $U_\lambda$ and $V_\lambda$ are eigenfunctions of the 
  real-valued Ornstein-Uhlenbeck generator and thus multiple Wiener-It\^o
  integrals. In contrast to the real case, however, one has a finer
  decomposition of the eigenspaces with a rich structure: For each $\lambda \in
  \N_0$, it holds that 
  \begin{equation*}
    \ker \left( \LL_{\text{OU}} + \lambda \Id \right) = \bigoplus_{\substack{p,q
      \in \N_0\\ p+q=\lambda}} \mathcal{H}_{p,q},
  \end{equation*}
  where the sum on the right is orthogonal and the spaces $\mathcal{H}_{p,q}$
  consist of complex Wiener-It\^o integrals of the form $I_{p,q}(f)$
  (see~\cite{ito_complex_1952}). Furthermore, one can show that
  $\overline{\mathcal{H}_{p,q}}=\mathcal{H}_{q,p}$ and that only the eigenspaces
  of even eigenvalues $\lambda=2p$ contain real-valued eigenfunctions, belonging
  to $\mathcal{H}_{p,p}$.
  Let us briefly outline the construction of an
  orthonormal basis for $\mathcal{H}_{p,q}$ (see again~\cite{ito_complex_1952}
  for details). For integers $p,q \geq 0$, the complex Hermite polynomials
  $H_{p,q}$ are given by 
  \begin{align*}
    H_{p,q}(z) &= (-1)^{p+q} \me^{\abs{z}^2} \left( \partial_z \right)^p
                 \left( \partial_{\overline{z}} \right)^q \me^{- \abs{z}^2}
    \\ &=
         \sum_{j=1}^{p \land q} \binom{p}{j} \binom{q}{j} j! (-1)^j z^{p-j}
         \overline{z}^{q-j}, 
  \end{align*}
  where summation ends at the smaller of the two parameters $p$ and $q$. 
  Now let $\left\{ e_j \colon j \geq 1 \right\}$ be an orthonormal basis of the
  underlying complex Hilbert space $H$ and $\left\{ Z(h) \colon h \in H
  \right\}$ denote the complex isonormal Gaussian process.
  Furthermore, for $n \in \N_0$, denote by $M_{n}$ the set of all multi-indices
  of order $n$ (sequences with a finite number of positive non-zero entries
  which   sum up to $n$) and, for $(m_p,m_q) \in M_p \times M_q$, define 
  \begin{equation*}
   \Phi_{m_p,m_q} = \prod_{j=1}^{\infty} H_{m_p(j),m_q(j)}(Z(e_j)).
  \end{equation*}
  Then, the family $\left\{ \Phi_{m_p,m_q} \colon (m_p,m_q) \in M_p \times M_q
  \right\}$ is an orthonormal basis of $\mathcal{H}_{p,q}$.
  In particular, as $H_{p,0}=z^p$, we see that for any multi-index $m \in
  M_{p}$, the monomial $\prod_{j=1}^{\infty} Z(e_j)^{m(j)}$ is an element of
  $\mathcal{H}_{p,0}$ and remains an eigenfunction when taking powers. In other
  words, for these eigenfunctions the Wick product coincides with the ordinary
  product. In the real case, there exist no non-constant eigenfunctions with
  this property.}
\end{example}

\section{Fourth Moment Theorems for complex Gaussian approximation}
\label{s-now-let-z}

Throughout the whole section, $\LL$ denotes a complex symmetric Markov diffusion
generator with invariant measure $\mu$ and discrete spectrum $S = \Set{ -
  \lambda_k}$, acting on $L^2(E,\C,\mu)$, where $E$ is a good measurable
space. The associated carr\'e du champ operator, acting on $\mathcal{A}_0$ which
{we assume to be} dense in $L^p(E,\C,\mu)$ for all $p \geq 1$, is
denoted by $\Gamma$.  

We start by introducing the notion of chaos, which for the real case was given
in~\cite{azmoodeh_fourth_2014} and extended to the multidimensional case
in~\cite{campese_multivariate_2015}. These definitions { can be
  generalized as follows to the complex setting.}

\begin{definition}\hfill
  \label{def:1}
  \begin{enumerate}[(i)]
  \item  Two eigenfunctions $F_1 \in \ker \left(\LL +\lambda_{p_1} \on{Id}\right)$ and $F_2 \in
    \ker \left(\LL +\lambda_{p_2} \on{Id}\right)$ are called \emph{jointly
      chaotic}, if $F_1F_2$, $F_1 \overline{F}_2$ (and thus, as the eigenspaces
    are closed under conjugation, also $\overline{F_1F_2}$ and
    $\overline{F}_1F_2$) have an expansion over the first $p_1+p_2+1$
    eigenspaces. In formulas, we require that
  \begin{equation*}
    F_1 F_2 \in \bigoplus_{k=0}^{p_1+p_2} \ker \left(\LL +\lambda_{k} \on{Id}
    \right)
    \qquad \text{and} \qquad
    F_1 \overline{F}_2 \in \bigoplus_{k=0}^{p_1+p_2} \ker \left(\LL +\lambda_{k}
      \on{Id} \right).
    \end{equation*}
  \item  A single eigenfunction is called chaotic, if it is jointly chaotic
    with itself.
\item  A vector $F=(F_1,\dots,F_d)$ of eigenfunctions $F_j \in \ker \left(\LL +
    \lambda_{p_j} \on{Id}\right)$ is called \emph{chaotic}, if any two of its
  components are \emph{jointly chaotic} (in particular, each component is
  \emph{chaotic} in the sense of part (ii)). 
  \end{enumerate}
\end{definition}

{Note that indeed, by taking all involved eigenfunctions to be real
  valued, we obtain the corresponding notions of real Markov chaos 
  (namely~\cite[Def. 2.2]{azmoodeh_fourth_2014})   
  and~\cite[Def. 3.2]{campese_multivariate_2015}) as special cases of
  Definition~\ref{def:1}.} 
As in the real case, a crucial ingredient for our main results will be the
following general principle. The proof for the real case
(see~\cite[Thm. 2.1]{azmoodeh_fourth_2014}) can be straightforwardly generalized
to the complex case and is therefore omitted.

\begin{theorem}
  \label{thm:3}
  Let $F \in \bigoplus_{k=0}^p \on{ker} \left( \LL + \lambda_k \Id
  \right)$. Then, for any~$\eta \geq \lambda_p$ it holds that
  \begin{equation*}
    \int_E^{} F \left( \LL + \eta \Id \right)^2 \overline{F} \diff{\mu}
    \leq
    \eta
        \int_E^{} F \left( \LL + \eta \Id \right) \overline{F} \diff{\mu}
        \leq
        c
    \int_E^{} F \left( \LL + \eta \Id \right)^2 \overline{F} \diff{\mu},
  \end{equation*}
  where $1/c$ is the minimum of the set $\Set{ \eta - \lambda_k \mid 0 \leq k
    \leq p} \setminus \Set{ 0}$. 
\end{theorem}

{Again, we note that by specializing to real valued eigenfunctions,
  we obtain~\cite[Thm. 2.1]{azmoodeh_fourth_2014} as a special case.} From
Theorem~\ref{thm:3}, we immediately deduce the following corollary. 

\begin{corollary}
  \label{cor:1}
  For two jointly chaotic eigenfunctions $F_{1} \in 
\ker \left( \LL + \lambda_{p_1} \Id \right)$ and $F_{2} \in
\ker \left( \LL + \lambda_{p_2} \Id \right)$ it holds that
\begin{equation}
  \label{eq:33}
  \int_E^{} \abs{\Gamma(F_1,F_2)}^2 \diff{\mu} \leq \frac{p_1+p_2}{2} \int_E^{}
  \overline{F}_1 F_2 \Gamma(F_1,F_2) \diff{\mu}.
\end{equation}
\end{corollary}

\begin{proof}
  By definition, $2 \Gamma(F_1,F_2) = \left( \LL + (p_1+p_2)\Id
\right) (F_1\overline{F}_2)$ and thus, using the fact that $\LL$ and the
identity are both self-adjoint and then Theorem~\ref{thm:3}, it follows that
\begin{align*}
  \int_E^{} \abs{\Gamma(F_1,F_2)}^2 \diff{\mu}
  &= \frac{1}{4} \int_E^{} \overline{F}_1 F_2
  \left( \LL + 
      (p_1+p_2) \Id \right)^2 (F_1 \overline{F}_2) \diff{\mu}
  \\  &\leq
  \frac{p_1+p_2}{4} \int_E^{} \overline{F}_1F_2
  \left( \LL + 
     (p_1+p_2) \Id \right) (F_1 \overline{F}_2) \diff{\mu}
 \\  &=
  \frac{p_1+p_2}{2} \int_E^{} \overline{F}_1F_2
  \Gamma(F_1,F_2) \diff{\mu}.
\end{align*}
\end{proof}

Before continuing, we need to introduce some notation.
For $d\geq 1$, let $F=(F_1,\dots,F_d)$ and $G=(G_1,\dots,G_d)$ be two complex
random vectors. The Wasserstein distance $d_W(F,G)$ between $F$ and $G$ is then
defined as
\begin{equation}
  d_W(F,G) = \sup_{h \in \mathcal{H}} \abs{ \Ex{h(F)} - \Ex{h(G)}},
\end{equation}
where $\mathcal{H} = \Set{ h \colon \C^d \to \C \mid \norm{h}_{\text{Lip}} \leq
  1 }$ and $\norm{h}_{\text{Lip}}$ denotes the Lipschitz norm, defined as
\begin{equation*}
  \norm{h}_{Lip} = \sup_{w,z \in \C^d} \frac{\abs{h(w) - h(z)}}{\norm{w-z}_2} =
  \sup_{w,z \in \C^d} \frac{\abs{h(w) -h(z)}}{\sqrt{\sum_{j=1}^d \abs{w_j-z_j}^{2}}}.
\end{equation*}
Furthermore, we will use the shorthand $\Gamma(F,\LL^{-1}G)$ to denote the
matrix $\left( \Gamma(F_j,\LL^{-1}G_k) \right)_{1 \leq j,k \leq d}$. 

The following Theorem provides a
quantitative bound on the Wasserstein distance between a complex random vector
and a multivariate complex Gaussian.  

\begin{theorem}
  \label{thm:1}
  For $d\geq 1$, let $Z \sim C\mathcal{N}_d(0,\Sigma)$ and denote by
  $F=(F_1,\dots,F_d)$ a complex random vector whose components are elements of
  $\mathcal{A}_0$. Then it holds that
  \begin{multline}
    \label{eq:5}
    d_W(F,Z)
    \leq 2  \norm{\Sigma^{-1}}_{\text{op}} \norm{\Sigma}_{\text{op}}^{1/2}
     \left(
         \int_E^{}
         \norm{\Gamma(\overline{F},-\LL^{-1} F)}_{\text{HS}}^2 \diff{\mu}
         \right.
         \\ +\left. \int_E^{}
          \norm{\Gamma(F,-\LL^{-1} F) - \Sigma}_{\text{HS}}^2 \diff{\mu}
          \right)^{1/2}.
        \end{multline}
\end{theorem}

{
\begin{remark}
  In the case of the real Ornstein-Uhlenbeck generator, a similar bound was
  obtained in~\cite[Theorem 3.3]{nourdin_multivariate_2010} through the use of
  Malliavin calculus. Formulated in the language of Markov diffusion generators,
  this bound reads
  \begin{equation*}
    d_W(U,W)
    \leq  \norm{C^{-1}}_{\text{op}} \norm{C}_{\text{op}}^{1/2}
     \left(
        \int_E^{}
          \norm{\Gamma(U,-\LL^{-1} U) - C}_{\text{HS}}^2 \diff{\mu}
          \right)^{1/2},
  \end{equation*}
  where $U$ is a real-valued, smooth random vector and $W \sim
  \mathcal{N}_d(0,C)$ a $d$-dimensional centered real Gaussian vector. Compared
  to the bound~\eqref{eq:5}, we see that in the complex setting a second
  $\Gamma$-term appears. 
\end{remark}
}

\begin{proof}[Proof of Theorem~\ref{thm:1}]
  Let $h \in C^2(\C^d)$ with bounded first and second derivatives and denote
  by $U_h$ the solution~\eqref{eq:4} to the Stein equation of Lemma~\ref{lem:3}. 
  By the integration by parts formula and the chain rule for $\Gamma$, it holds
  that
  \begin{align*}
    \int_E^{} \left\langle \nabla U_h(F),\overline{F} \right\rangle_{\C^d} \diff{\mu}
    &=
    \sum_{k=1}^d \int_E^{} \left(\partial_{z_k} U_h(F)\right) F_k \diff{\mu}
    \\ &=
    \sum_{k=1}^d \int_E^{} \left(\partial_{z_k} U_h(F)\right)
         \LL \LL^{-1}F_k \diff{\mu}
    \\ &=
    \sum_{k=1}^d \int_E^{}  \Gamma \left(\partial_{z_k} U_h(F),
         -\LL^{-1}\overline{F}_k\right) \diff{\mu}
    \\ &=
    \sum_{j,k=1}^d \int_E^{}  \Big( \partial_{z_jz_k} U_h(F) \Gamma \left( F_j,
         -\LL^{-1}\overline{F}_k\right)
    \\ & \qquad \qquad \qquad \qquad +
         \partial_{\overline{z}_jz_k} U_h(F) \Gamma \left( \overline{F}_j,
         -\LL^{-1}\overline{F}_k\right)
         \Big) \diff{\mu}
    \\ &=
         \int_E^{}
         \left\langle \nabla\nabla U_h(F), \Gamma(\overline{F},-\LL^{-1} F)
         \right\rangle_{\text{HS}} \diff{\mu}
    \\ &\qquad \qquad \qquad \qquad+
         \int_E^{}
         \left\langle \overline{\nabla}\nabla U_h(F), \Gamma(F,-\LL^{-1} F)
         \right\rangle_{\text{HS}}
         \diff{\mu}.
  \end{align*}
  Analogously 
  \begin{align*}
    \int_E^{} \left\langle \overline{\nabla} f(F),F
    \right\rangle_{\C^d} \diff{\mu}
    &=
         \int_E^{}
         \left\langle \overline{\nabla\nabla} f(F), \Gamma(F,-\LL^{-1} \overline{F})
         \right\rangle_{\text{HS}} \diff{\mu}
    \\ &\qquad \qquad \qquad \qquad+
         \int_E^{}
         \left\langle \nabla \overline{\nabla} f(F), \Gamma(\overline{F},-\LL^{-1} \overline{F})
         \right\rangle_{\text{HS}}
         \diff{\mu}.
  \end{align*}
  Plugging these two identities into the complex Stein equation yields
  \begin{align*}
\notag
   \int_E^{} h(F) \diff{\mu} - \Ex{h(Z)} &=
    \int_E^{}
         \left\langle \nabla\nabla U_h(F), \Gamma(\overline{F},-\LL^{-1} F)
         \right\rangle_{\text{HS}} \diff{\mu}
    \\ &\qquad+         \notag 
         \int_E^{}
         \left\langle \overline{\nabla\nabla} U_h(F), \Gamma(F,-\LL^{-1} \overline{F})
         \right\rangle_{\text{HS}} \diff{\mu}
    \\ & \qquad+ \notag
         \int_E^{}
         \left\langle \nabla\overline{\nabla} U_h(F), \Gamma(F,-\LL^{-1} F) - \Sigma
         \right\rangle_{\text{HS}}
         \diff{\mu}
    \\ &\qquad+ \notag
         \int_E^{}
         \left\langle \overline{\nabla} \nabla U_h(F),
         \Gamma(\overline{F},-\LL^{-1} \overline{F} - \overline{\Sigma})
         \right\rangle_{\text{HS}}
         \diff{\mu}
    \\ &= I_1 + I_2 + I_3 +I_4,
  \end{align*}
  so that
  \begin{align}
    \notag
    \abs{   \int_E^{} h(F) \diff{\mu} - \Ex{h(Z)}}
    &= \sqrt{\abs{I_1+I_2+I_3+I_4}^2}
    \\ \label{eq:11}  &\leq 2
    \sqrt{\abs{I_1}^2 + \abs{I_2}^2 + \abs{I_3}^2 + \abs{I_4}^2}.
  \end{align}

  Using Lemma~\ref{lem:5}, we obtain
  \begin{align*}
    \abs{I_1}^2+\abs{I_3}^2
    &\leq
      \int_E^{} \left(
      \abs{         \left\langle \nabla\nabla U_h(F),
      \Gamma(\overline{F},-\LL^{-1} F)  \right\rangle_{\text{HS}}}^2 \right.
    \\ &\qquad \qquad \qquad \qquad+
         \left.
         \abs{\left\langle \overline{\nabla\nabla} U_h(F), \Gamma(F,-\LL^{-1}
         \overline{F}) 
         \right\rangle_{\text{HS}}}^2
               \right)      \diff{\mu}
                              \\ &\leq
      \norm{\Sigma^{-1}}_{\text{op}}^2 \norm{\Sigma}_{\text{op}}
                    \norm{h}_{\text{Lip}}^2
                            \int_E^{}
         \norm{\Gamma(\overline{F},-\LL^{-1} F)}_{\text{HS}}^2
         \diff{\mu} 
    \\
    \intertext{and similarly}
    \abs{I_2}^2+\abs{I_4}^2 &\leq
                              \norm{\Sigma^{-1}}_{\text{op}}^2
                    \norm{\Sigma}_{\text{op}}  \norm{h}_{\text{Lip}}^2
            \int_E^{} \norm{\Gamma(F,-\LL^{-1} F) - \Sigma}_{\text{HS}}^2 \diff{\mu}.
  \end{align*}
  Plugged back into~\eqref{eq:11}, this gives
  \begin{multline*}
    \abs{ \int_E^{}h(F) \diff{\mu} - \Ex{h(Z)}}
    \\ \leq 2
                  \norm{\Sigma^{-1}}_{\text{op}} \norm{\Sigma}_{\text{op}}^{1/2}
      \norm{h}_{\text{Lip}} 
         \Big(
         \int_E^{}
         \norm{\Gamma(\overline{F},-\LL^{-1} F)}_{\text{HS}}^2
         \diff{\mu}
     \\ +
         \int_E^{} \norm{\Gamma(F,-\LL^{-1} F) - \Sigma}_{\text{HS}}^2 \diff{\mu}
          \Big)^{1/2},
  \end{multline*}
  where $h \in \mathcal{C}^2(\C^d,\C)$ with bounded first and second
  derivatives. The proof is finished by noting that any Lipschitz function $g$
  can be uniformly approximated by functions of this type (take for example
  $g_{\varepsilon}(z) = \Ex{g(z + \sqrt{\varepsilon} Z)}$, where $Z \sim
  C\mathcal{N}_d(0,\Id_d)$; see~\cite[proof of Lemma
  3.3]{nourdin_multivariate_2010}). 
  \end{proof}

  {
  In the framework of real Markov diffusion generators, one can obtain bounds
  for the stronger Kolmogorov and total variation distances when specializing to 
  dimension one. As a complex random variable corresponds to a two-dimensional
  real random vector, this strengthening is of course no longer possible using
  this approach.}    

  For chaotic complex random vectors, the integrals appearing in the
  bound~\eqref{eq:5} of Theorem~\ref{thm:1} can be expressed {purely} in terms of
  moments as follows.

  \begin{theorem}
    \label{thm:4}
    For $d\geq 1$, let $Z \sim C\mathcal{N}_d(0,\Sigma)$, where
    {$\Sigma=(\sigma_{j,k})_{1 \leq j,k \leq d}$}, and
    $F=(F_1,\dots,F_d)$ 
    be a chaotic complex random vector. Then it holds that
    \begin{equation}
      \label{eq:21}
      d_{W}(F,Z) \leq 
                        \norm{\Sigma^{-1}}_{\text{op}}
                        \norm{\Sigma}_{\text{op}}^{1/2} 
      \sqrt{\Psi_1(F) + \Psi_2(F) + \Psi_3(F)},
    \end{equation}
    where
    \begin{align*}
      \Psi_1(F) &=  \sum_{j,k=1}^d \abs{
                  \int_E^{} F_j \overline{F}_k \diff{\mu} - \sigma_{j,k}}^2
     \\             
      \Psi_2(F) &= \sum_{j,k=1}^d
                  \sqrt{
                  \int_E^{} \abs{F_j}^4 \diff{\mu}
                  \left( \frac{1}{2} \int_E^{} \abs{F_k}^4 \diff{\mu} -
                  \left( \int_E^{} \abs{F_k}^2 \diff{\mu} \right) \right)
                  }
                  \intertext{and}
                  \Psi_3(F) &= \sum_{j,k=1}^d
                              \int_E^{} \abs{F_jF_k}^2 \diff{\mu} - \int_E^{}
                              \abs{F_j}^2 \diff{\mu} \int_E^{} \abs{F_k}^2
                    \diff{\mu}
                              -
                              \abs{
                              \int_E^{} F_j \overline{F}_k \diff{\mu} }^2.
    \end{align*}
  \end{theorem}

  \begin{proof}
    In view of Theorem~\ref{thm:1}, we have to show that
    \begin{multline*}
      \int_E^{}
      \left(
         \norm{\Gamma(\overline{F},-\LL^{-1} F)}_{\text{HS}}^2
     +
         \norm{\Gamma(F,-\LL^{-1} F) - \Sigma}_{\text{HS}}^2 \right) \diff{\mu}
      \\ \leq \Psi_1(F) + \Psi_2(F) + \Psi_3(F).
    \end{multline*}
  When expanding the two Hilbert-Schmidt norms, the integral on the left hand
  side becomes
  \begin{equation}
    \label{eq:43}
    \sum_{j,k=1}^d
    \int_E^{} \left(
         \abs{\Gamma(\overline{F}_j,-\LL^{-1}F_k)}^2
         +
         \abs{\Gamma(F_j,-\LL^{-1}F_k) - \sigma_{j,k}}^2
         \right)
    \diff{\mu}
  \end{equation}
  Now note that by integration by parts and Corollary~\ref{cor:1},
  \begin{align}
    \int_E^{}\abs{\Gamma(F_j,-\LL^{-1}F_k) - \sigma_{j,k}}^2 \diff{\mu} \notag
    &=
         \int_{E} \abs{\Gamma(F_j,-\LL^{-1}F_k)}^2 \diff{\mu} + \abs{\sigma_{j,k}}^2
    \\ &\qquad - \notag
         2
         \mathfrak{Re} 
    \left( \int_E^{}\Gamma(F_j,-\LL^{-1}F_k) \diff{\mu} \,
         \overline{\sigma}_{j,k}  \right)
    \\ &= \notag
        \int_{E} \abs{\Gamma(F_j,-\LL^{-1}F_k)}^2 \diff{\mu} +
         \abs{\sigma_{j,k}}^2
     \\ &\qquad  \notag          -   2
         \mathfrak{Re} 
    \left( \int_E^{}F_j \overline{F}_k \diff{\mu} \,
          \overline{\sigma}_{j,k}  \right)
    \\ &= \notag
         \int_{E} \abs{\Gamma(F_j,-\LL^{-1}F_k)}^2 \diff{\mu} 
         - \abs{\int_E^{} F_j \overline{F}_k \diff{\mu}}^2
    \\ &\qquad \notag
         + \abs{ \int_E^{} F_j \overline{F_k} \diff{\mu} - \sigma_{j,k}}^2
    \\ &\leq \label{eq:40}
         \int_{E} \overline{F}_j F_k\Gamma(F_j,-\LL^{-1}F_k) \diff{\mu} 
         - \abs{\int_E^{} F_j \overline{F}_k \diff{\mu}}^2
    \\ &\qquad \notag
         + \abs{ \int_E^{} F_j \overline{F_k} \diff{\mu} - \sigma_{j,k}}^2.
  \end{align}
  On the other hand, by~Corollary~\ref{cor:1}, the chain rule and integration by parts,
  \begin{align}
    \int_E^{} \abs{\Gamma(\overline{F}_j,-\LL^{-1}F_k)}^2 \diff{\mu}
    & \leq \notag
      \int_E^{} F_j F_k \Gamma(\overline{F}_j,-\LL^{-1}F_k) \diff{\mu}
    \\ &= \label{eq:42}
         \int_E^{} \Gamma(F_j\overline{F}_jF_{k},-\LL^{-1}F_k) \diff{\mu}
    \\ & \qquad \qquad - \notag
         \int_E^{} \abs{F_j}^{2} \Gamma(F_{k},-\LL^{-1}F_k)
         \diff{\mu}
    \\ & \qquad \qquad - \notag
     \int_E^{} \overline{F}_jF_{k} \Gamma(F_j,-\LL^{-1}F_k)
         \diff{\mu}
  \end{align}
  Plugging~\eqref{eq:40} and~\eqref{eq:42} into~\eqref{eq:43} yields that
  \begin{multline*}
    \sum_{j,k=1}^d
    \int_E^{} \left(
         \abs{\Gamma(\overline{F}_j,-\LL^{-1}F_k)}^2
         +
         \abs{\Gamma(F_j,-\LL^{-1}F_k) - \sigma_{jk}}^2
         \right)
    \diff{\mu}
    \\ \leq
    \Psi_{1} - \sum_{j,k=1}^d \int_E^{} \abs{F_j}^2 \left( \Gamma(F_k,-\LL^{-1}F_k) -
    \int_E^{} \abs{F_k}^2 \right) \diff{\mu} + \Psi_3.
  \end{multline*}
  To see that the sum in the middle is bounded by $\Psi_2(F)$, we apply
  Cauchy-Schwarz to each summand and then make use of the complex Gamma calculus
  once more to transform the remaining $\Gamma$-expression into a moment:
  \begin{align*}
    \int_E^{} \Gamma(F_k,-\LL^{-1}F_k)^2 \diff{\mu}
    &\leq
      \int_E^{} F_k \overline{F}_k \Gamma(F_k,-\LL^{-1}F_k) \diff{\mu}
    \\ &=
         \frac{1}{2}
         \int_E^{} \Gamma(F_k^2\overline{F_k},-\LL^{-1}F_k) \diff{\mu}
    \\ &\qquad \qquad  -
         \int_E^{} F_k^2 \Gamma(\overline{F_k},-\LL^{-1}F_k) \diff{\mu}
    \\ &\leq
    \frac{1}{2}
         \int_E^{} \Gamma(F_k^2\overline{F_k},-\LL^{-1}F_k) \diff{\mu}
    \\ &=
    \frac{1}{2} \int_E^{} \abs{F_k}^4 \diff{\mu},
  \end{align*}
  {where the last inequality follows from Corollary~\ref{cor:1}, which implies
  that $\int_E^{}F_k^2 \Gamma(\overline{F}_k,-\LL^{-1}F_k)\diff{\mu} \geq 0$.}
\end{proof}

For eigenfunctions of the real Ornstein-Uhlenbeck generator, a bound of a
similar type has been proven in~\cite[Theorem 1.5]{noreddine_gaussian_2011}
using Malliavin calculus{, with the notable difference that only
  non-mixed fourth moments appear. The same strategy could be followed to
  prove a refined version of the bound~\eqref{eq:21} for eigenfunctions of the
  complex Ornstein-Uhlenbeck generator (i.e. complex multiple Wiener-It\^o
  integrals; see Example~\ref{ex:1}), exclusively involving the second moments
  $\int_E^{} F_j \overline{F}_k \diff{\mu}$, $1 \leq j,k \leq d$ and the
  non-mixed fourth moments $\int_E^{} \abs{F_j}^4 \diff{\mu}$, $1 \leq j \leq
  d$.} 

Applying the Gaussian integration by parts formula, one sees for~$Z \sim
C\mathcal{N}_d(0,\Sigma)$ and $j=1,2,3$ that indeed $\Psi_j(Z)=0$. Therefore, we
have the following corollary.

\begin{corollary}
  \label{cor:3}
  For $d \geq 1$, let $Z \sim C\mathcal{N}_d(0,\Sigma)$ and
  $F_n=(F_{1,n},\dots,F_{d,n})$ 
  be a sequence of centered chaotic complex random vectors.
  Then, $(F_n)$ converges in distribution towards $Z$, if, and only if,
  \begin{equation}
    \label{eq:44}
    \int_E^{} F_{j,n} \overline{F}_{k,n} \diff{\mu} \to \Ex{ Z_j \overline{Z}_k}
  \end{equation}
  and
  \begin{equation}
    \label{eq:6}
    \int_E^{} \abs{F_{j,n} F_{k,n}}^2 \diff{\mu}  \to \Ex{ \abs{Z_j Z_k}^2}
  \end{equation}
  for $1 \leq j,k \leq d$.  
\end{corollary}

\begin{remark}\hfill
  \begin{enumerate}[(i)]
  \item For $d=1$ and $\Sigma=\sigma^2>0$, Corollary~\ref{cor:3} says that a
    sequence of centered chaotic eigenfunctions converges in distribution
    towards a one-dimensional centered complex Gaussian random variable with
    variance~$\sigma^2$, if, and only if, its second and fourth absolute
    moments converge towards $\sigma^2$ and 
    $2\sigma^4$, respectively. This is the complex counterpart
    of the abstract Fourth Moment Theorem for Gaussian approximation
    (\cite[Corollary 3.3]{azmoodeh_fourth_2014}). If, in addition, we take $\LL$
    to be the complex Ornstein-Uhlenbeck generator (see Example~\ref{ex:1}), we
    obtain~\cite[Theorem 1.1.1]{chen_fourth_2014}).
 \item For $d \geq 2$, Corollary~\ref{cor:3} is the complex counterpart
   of~\cite[Theorem 1.2]{campese_multivariate_2015}.
  \end{enumerate}
\end{remark}

If $\LL$ is the real Ornstein-Uhlenbeck generator, the Peccati-Tudor Theorem
(\cite{peccati_gaussian_2005}) says that a centered sequence $(F_{n})$ of
vectors of eigenfunctions of $\LL$ (i.e. multiple integrals) converges jointly
in distribution towards a centered Gaussian random vector with covariance
$\Sigma$, if, and only if, $\on{Var}(F_n) \to 0$ and each component sequence
converges separately towards a (one-dimensional) Gaussian. This result has been
generalized in~\cite[Proposition 3.5]{campese_multivariate_2015} to the abstract
diffusion generator framework. A straightforward adaptation of {the latter
finding} yields the following complex Peccati-Tudor Theorem:

\begin{proposition}
  For $d \geq 2$, let $Z\sim C\mathcal{N}_d(0,\Sigma)$, where $\Sigma$ is
  positive definite and Hermitan, and let $F_n=(F_{1,n},\dots,F_{d,n})$ be a
  centered chaotic vector whose covariance converges towards $\Sigma$ as $n \to
  \infty$. Furthermore, assume that
  \begin{enumerate}[1.]
  \item  The underlying generator $\LL$ is ergodic, in the sense that its kernel
    only consists of constants
  \item  If, for $1 \leq j < k \leq d$ and $j \neq k$, the pair
    $(F_{j,n},F_{k,n})_n$ has a subsequence $(F_{j,n_{l}},F_{k,n_l})_l$ such
    that $F_{j,n_l}$ and $F_{k,n_l}$ are elements of the same eigenspace with
    eigenvalue $\lambda_{l}$ for all $l$, it holds that 
  \begin{equation*}
    \int_E^{} \pi_{2\lambda_l}(F_{j,n_l}^2)
    \pi_{2\lambda_l}(\overline{F}_{j,n_l}^2) \diff{\mu}
    -
    2 \left( \int_E^{} F_{j,n_l} \overline{F}_{j,n_l} \diff{\mu} \right)^2 \to 0,
  \end{equation*}
  where $\pi_{\lambda}$ denotes the orthogonal projection onto $\ker
  \left(L+\lambda\Id\right)$.
  \end{enumerate}
  Then, the following two assertions are equivalent.
  \begin{enumerate}[(i)]
  \item $F_n \xrightarrow{d} Z$.
  \item For $1 \leq j \leq d$ it holds that $F_{j,n} \xrightarrow{d} Z_j$.
  \end{enumerate}
\end{proposition}

{
In~\cite{azmoodeh_fourth_2014}, Fourth Moment Theorems for the Gamma and
Beta distribution were derived. We would like to mention that our techniques
could be readily applied to extend these results and cover
target random variables whose real and imaginary parts are independent real
Gamma or Beta random variables (in the Ornstein-Uhlenbeck case, this has been
done for the Gamma case in~\cite{chen_fourth_2014} by treating real and
imaginary part separately).
}

{
\section{Acknowledgements}
\label{s--begintheorem-}
The author thanks Domenico Marinucci for many useful remarks and a careful
reading of an earlier version of this paper.
}

\bibliography{refs_new}
\end{document}